\documentclass[11pt,fleqn]{article}
\usepackage{amsmath, amscd,amsbsy, amsthm, amsfonts, amssymb,array}
\usepackage{rotating}
\usepackage{dcolumn,multicol,url}
\usepackage[all]{xy}
 \newtheorem{thm}{Theorem}[section]
 \newtheorem{cor}[thm]{Corollary}
 \newtheorem{lem}[thm]{Lemma}

 \theoremstyle{definition}

 \numberwithin{equation}{section}

 \newcommand{\A}{\mathcal{A}}

\def\be{\begin{equation}}
\def\ee{\end{equation}}
\def\ba{\begin{eqnarray}}
\def\ea{\end{eqnarray}}

\def\e1{\epsilon}
\def\AAl{\mathcal{A}_{\lambda}}
\def\A0{\stackrel{\circ}{\AAl}}

\def\o1{\omega}
\def\01{\Omega}
\def\c1{\gamma}

\def\g1{\Sigma}

\def\l1{\Lambda}

\def\v1{\varphi}

\def\d1{\delta}

\def\f1{\frac}
\def\t1{\theta}
\def\b1{\beta}

\def\bs{\begin{eqnarray*}}
\def\es{\end{eqnarray*}}
\def\m1{\Theta}
\def\w1{\wedge}

\title{Correction to Section 19.2 of Ricci Flow and the Poincar\'e Conjecture}
\author{John W. Morgan and Gang Tian}

\begin{document}

\maketitle

Abbas Bahri and Terry Tao have pointed out to us that the estimates we claimed in Chapter 19 of our book \cite{MT}  ``Ricci Flow and the Poincar\'e Conjecture" (the book) for the curve shrinking flow in an ambient Ricci flow are not correct.
The main error occurs 
on page 442 where we made a mistake in computing evolution of the curvature of the curve. Specifically, in the equation in  Lemma 19.7, we missed terms related to the time derivative of connections of $g(t)$. From this incorrect equation we deduced a bound of the form
$$\frac{\partial}{\partial t}k^2 \le \text{\rm good terms we can handle  } + Ck^2,$$
where $k$ is the curvature of the curve and $C$ is a constant depending only on the ambient Ricci flow. 
After taking into account 
the time derivative of connections of g(t), we see that there is  an extra term in the upper bound of the form  $C' k$.
 The presence of terms linear
in $k$ in this inequality translates to a constant term in the integrand of the upper bound of the time derivative of the total curvature, meaning that the exponential bound for the evolving total curvature has a multiplicative constant term that depends on both the initial total curvature  and the initial length, rather than just on the initial total curvature as stated in Corollary 19.10.
This has no effect on the application later in Chapter 19, an application to a compact family of curve-shrinking flows of $C^2$-curves, because we have uniform bounds on both total curvature and length at the initial time and hence uniform bounds on length and total curvature at all later times.

In his treatment of this material in \cite{P}, Perelman states correctly that the upper bound contains terms of the form $C(k^2+k)$ without giving an  explicit formula or derivation for the upper bound. 

We thank Bahri for his persistent belief that these inequalities in the book were incorrect and to Tao for pointing out to us exactly which terms we overlooked.  

All references to numbered statements, equations, and pages refer to the book. What we describe here  is a replacement for the material on pages 441 through 445 and corrects the oversight described above. The way we approach the computation is to work with the `obvious' metric on $M\times I$ to derive the correct formula for the time derivative of $k^2$ and then to make minor adjustments in the subsequent discussion. 
 
 \bigskip
 
\noindent{\bf Preliminary Material.}
 The context is the following:  There is a Ricci flow $g(t)$ defined on a compact manifold $M$ for $t\in I$ where $I=[0,T]$.
 In $M\times I$ we use the terminology horizontal (for directions in $M$) and vertical for the $I$-direction. We denote by $\hat g$ the metric on $M\times I$ that is the orthogonal sum of the varying family $g(t)$ of horizontal metrics and the usual metric on $I$. We denote by $\nabla$ the covariant derivative associated to $\hat g$ and by $\nabla^g$ the varying family of horizontal covariant derivatives associated to the family of metrics $g(t)$ (an operator that only applies to horizontal vector fields). All inner products are taken with respect to $\hat g$. Of course, the inner product of horizontal vector fields agrees with the moving family of horizontal inner products associated with the $g(t)$ on these horizontal vector fields. Lastly, the Riemannian curvature tensor for $\hat g$ is denoted $\widehat{Rm}$
and the Ricci curvatures of the $g(t)$ are denoted ${\rm Ric}_g$.

Let $\widehat \Gamma$ be the Christoffel symbols for $\hat g$ and $\Gamma$ those for the moving family $g(t)$ of horizontal metrics. Denote by
$i,j,k,$ etc be coordinate indices in the horizontal direction, with $e_i,e_j,e_k,$ etc, being the corresponding coordinate vector fields, and by $0$  the coordinate direction $\partial_t$. Direct computation, using the fact that $\partial g/\partial t=-2Ric_g$ shows that:
\begin{align*}\widehat \Gamma^k_{ij} & =  \Gamma^k_{ij} \\
\widehat \Gamma^k_{i0} & =\widehat \Gamma^k_{0i}  =  -g^{k\ell}({\rm Ric}_g)_{i\ell} \\
\widehat \Gamma^0_{ij} & =  ({\rm Ric}_g)_{ij},
\end{align*}
and all other Christoffel symbols for $\hat g$ vanish.
This means that for horizontal vector fields $A,B$, we have 
\begin{align}\label{Eq1} \nabla_A(B) & =\nabla^g_A(B)+Ric_g(A,B)\partial_t \\ 
\langle \nabla_A(\partial_t),B\rangle & =-{\rm Ric}_g(A,B).\label{Eq2}
\end{align}
Also, $\nabla_A(\partial t)$ is horizontal.

Suppose that we have a map $c\colon S^1\times I\to M$ satisfying the curve-shrinking equation
$$\frac{\partial c(x,t)}{\partial t} = H(x,t),$$
where $H=H(x,t)$ is the curvature vector $H=\nabla^g_SS$, with $S$ being the unit tangent vector field 
in the increasing $x$-direction along the circle at point $(x,t)$. We denote the curvature of the curve, $|H|$, by $k$. The function $k$ is not
smooth at points where it is zero but $k^2=\langle H,H\rangle$ is everywhere smooth. We define 
$$\hat c\colon S^1\times I\to M\times I$$
by $\hat c(x,t)=(c(x,t),t)$. The image, $\Sigma$, of $\hat c$ is an immersed surface since the family of curves is assumed to be a family of immersed curves and  the component of $\hat c_*(\partial_t)$ in the $I$-direction is of unit length.
We have the vector fields $X=c_*(\partial_x)$ and $\widehat H=c_*(\partial_t)$ in $M\times I$ tangent to $\Sigma$. 
They are coordinate vector fields on the surface and hence commute. The first, $X$, is horizontal. Denote by $H$  the horizontal component of $\widehat H$; its vertical component is $\partial_t$. These are both vector fields defined on all of $\Sigma$, though they are not tangent to $\Sigma$.

We shall be taking covariant derivatives using the Levi-Civit\`a connection associated with the metric $\hat g$ along the surface $\Sigma$. This is possible when the direction of derivation is tangent to $\Sigma$ and the vector field being derived exists on $\Sigma$. We can also take the bracket of vector fields tangent to $\Sigma$ with the result being another vector field tangent to $\Sigma$.

\begin{lem}{\rm (Restated version of Lemma 19.6 in this notation.)}
With the above hypothesis we have
$$\widehat H(|X|^2)=-2{\rm Ric}_{g}(X,X)-2k^2|X|^2,$$
and
$$[\widehat H,S]=\bigl(k^2+{\rm Ric}_{g}(S,S)\bigr)S.$$
\end{lem}

\begin{proof}
N.B. The equations written here are what was implicitly meant in the statements of Lemma 19.6 in the book.

Since $\widehat H$ and $X$ commute, we have
$$\widehat H(|X|^2)=2\langle \nabla_{\widehat H}(X),X\rangle=2\langle \nabla_X\widehat H,X\rangle=2\langle \nabla_X\partial_t,X\rangle+2\langle \nabla_XH,X\rangle.$$

By Equation~(\ref{Eq2}) we have
$$
2\langle \nabla_X\partial_t,X\rangle = -2({\rm Ric}_g)(X,X).
$$
For the second term, since $X$ and $H$ are orthogonal
\begin{align*}
2\langle \nabla_XH,X\rangle & =-2\langle H,\nabla_XX\rangle \\
& =-2|X|^2\langle H,\nabla_SS\rangle-2\langle H,\nabla_X(|X|)S\rangle \\
& = -2k^2|H|^2.
\end{align*}

As for the second formula, the derivation on p 442 is valid. (What is called $H$ there is actually what is denoted $\widehat H$ here.)
\end{proof}

\bigskip

\noindent{\bf Corrections for Section 19.2.}  Next we compute a replacement for Lemma 19.7 of p 442. This is the equation that has extra terms overlooked in the book.
\begin{lem}{\rm (Corrected version of 19.7.)}\label{lem1}
Let $s$ denote an arc-length parameter on $\hat c(x,t)$. We shall denote derivatives with respect to $s$ by $'$. We have:
\begin{align*}
\frac{\partial}{\partial t}k^2  = & (k^2)''-2\langle( \nabla_SH)^\perp,(\nabla_SH)^\perp\rangle+2k^4 \\
&-2{\rm Ric}_g(H,H)+4k^2{\rm Ric}_g(S,S) +2\widehat{Rm}(\widehat H,S,H,S)\\
& +2{\rm Ric}_g(S,S){\rm Ric}_g(H,H) -2(\nabla^g_S({\rm Ric}_g))(S,H).
\end{align*}
where the superscript $\perp$ means the component orthogonal to $X$.
\end{lem}

\noindent
N.B. The last two terms  in this equation do not appear in Lemma 19.7
and the Riemannian curvature term here differs from the one in Lemma 19.7 in two respects: it is the Riemannian curvature of the metric $\hat g$ rather than of $g$, and one of the vector fields on which it is evaluated is
$\hat H$ replacing $H$ in Lemma 19.7.
 
\begin{proof}
First of all we have
$$\frac{\partial}{\partial t}k^2  =  \widehat H(\langle H,H\rangle).$$
By Equation~(\ref{Eq1}) we have
$\nabla_SS=H+{\rm Ric}_g(S,S)\partial_t$.
This gives
\begin{align*}
\frac{\partial}{\partial t}k^2 = &2\langle \nabla_{\widehat H}\nabla_SS,H\rangle -2\langle\nabla_{\widehat H}({\rm Ric}_g(S,S)\partial_t),H\rangle\\
=& 2\langle \nabla_S\nabla_{\widehat H}S,H\rangle+2\langle\nabla_{[\widehat H,S]}S,H\rangle+2\widehat{Rm}(\widehat H,S,H,S) \\
& -2\langle\nabla_{\widehat H}({\rm Ric}_g(S,S)\partial_t),H\rangle \\
=&  2\langle \nabla_S\nabla_S{\widehat H},H\rangle+2\langle \nabla_S[\widehat H,S],H\rangle+2\langle\nabla_{[\widehat H,S]}S,H\rangle \\ 
& +2\widehat{Rm}(\widehat H,S,H,S) -2\langle\nabla_{\widehat H}({\rm Ric}_g(S,S)\partial_t),H\rangle
\end{align*}
Consider $\langle\nabla_{\widehat H}({\rm Ric}_g(S,S)\partial_t),H\rangle$. Since $\langle\partial_t,H\rangle=0$, this term is equal to
${\rm Ric}_g(S,S)\langle\nabla_{\hat H}\partial_t, H\rangle$. Since $\nabla_{\partial_t}\partial_t=0$, using Equation~(\ref{Eq1}) we have
$$\langle\nabla_{\widehat H}({\rm Ric}_g(S,S)\partial_t),H\rangle={\rm Ric}_g(S,S)\langle\nabla_H\partial_t,H\rangle=-{\rm Ric}_g(S,S){\rm Ric}_g(H,H).$$

We break the first term on the right-hand side into 
\begin{equation}\label{Eq3}
2\langle\nabla_S\nabla_S\hat H,H\rangle=2\langle \nabla_S\nabla_S{\partial_t},H\rangle+2\langle \nabla_S\nabla_S{H},H\rangle.
 \end{equation}
According to Equation~(\ref{Eq2}) the vector field $\nabla_S(\partial_t)$ is a horizontal vector field along the surface
whose inner product with any horizontal vector field $Z$  is $-{\rm Ric}_g(S,Z)$.
This implies that $\nabla^g_S(\nabla_S(\partial_t))$ is a horizontal vector field along the surface whose inner product with any horizontal $Z$ is
$$-(\nabla_S^g({\rm Ric}_g))(S,Z) -{\rm Ric}_g(\nabla_S^gS,Z)=-(\nabla_S^g({\rm Ric}_g))(S,Z) -{\rm Ric}_g(H,Z).$$
Furthermore, since $\nabla_S(\partial_t)$ and $Z$ are horizontal,
$$\langle\nabla_S(\nabla_S(\partial_t)),Z\rangle=\langle\nabla^g_S(\nabla_S(\partial_t)),Z\rangle.$$
Applying this with $Z=H$ gives
$$2\langle \nabla_S\nabla_S{\partial_t},H\rangle=-2(\nabla_S^g({\rm Ric}_g))(S,H) -2{\rm Ric}_g(H,H).$$

 The second term in Equation~(\ref{Eq3}) can be rewritten as 
$$S(S(|H|^2)-2\langle\nabla_SH,\nabla_SH\rangle.$$
Since $S$ is a unit vector we have 
$$ \nabla_SH=(\nabla_SH)^\perp+\langle\nabla_SH,S\rangle S$$
where $\perp$ means the component orthogonal to $S$.
Thus, 
$$-2\langle\nabla_SH,\nabla_SH\rangle=-2\langle (\nabla_SH)^\perp,(\nabla_SH)^\perp\rangle-2\bigl(\langle \nabla_SH,S\rangle\bigr)^2.$$
Since $S$ and $H$ are orthogonal, it follows that
$$\langle \nabla_SH,S\rangle=-\langle H,\nabla_SS\rangle=-\langle H,H+{\rm Ric}_g(S,S)\partial_t\rangle=-\langle H,H\rangle=-k^2,$$
and hence that
$$-2\langle\nabla_SH,\nabla_SH\rangle=-2\langle (\nabla_SH)^\perp,(\nabla_SH)^\perp\rangle-2k^4.$$

Now we plug in for $[\widehat H,S]$
to get
$$\langle\nabla_S[\hat H,S],H\rangle=\langle \nabla_S\bigl((k^2+{\rm Ric}_g(S,S))S\bigr),H\rangle$$
and
$$\nabla_{[\hat H,S]}S,H\rangle=\langle\nabla_{(k^2+{\rm Ric}_g(S,S))S}S,H\rangle.$$
We have 
$$2\langle\nabla_{(k^2+({\rm Ric}_g)(S,S))S}S,H\rangle=2\bigl(k^2+{\rm Ric}_g(S,S)\bigr)\langle H,H\rangle=2\bigl(k^2+{\rm Ric}_g(S,S)\bigr)k^2.$$
Since $S$ and $H$ are perpendicular we have
$$2\langle \nabla_S(k^2+({\rm Ric}_g)(S,S))S,H\rangle=2\bigl(k^2+{\rm Ric}_g(S,S)\bigr)k^2.$$

Since $S(S(|H|^2))=(k^2)''$ all of these together establish the formula given in the lemma.
\end{proof}

Notice that the last five terms on the right-hand side of the equation in Lemma~\ref{lem1} are bounded by a constant depending only on the ambient Ricci flow,
(in fact only on the curvature and its first covariant derivatives), times
$k^2$, $k^2$, $k^2+k$, $k^2$, and $k$, respectively.
Thus, we have the following replacement for Equation (19.5) on p 443:
\begin{cor}
There is a constant $\widehat C<\infty$ depending only on the ambient Ricci flow such that
$$\frac{\partial}{\partial t}k^2 \le (k^2)''-2\langle \nabla_SH)^\perp,(\nabla_SH)^\perp\rangle+2k^4
+\widehat C(k^2+k).$$
\end{cor}

Now it is time to modify $k$ by replacing it with $h=h_\epsilon=\sqrt{k^2+\epsilon^2}$ for some $\epsilon>0$. 

Then $h>k$ and $h'=\frac{k}{h}k'$ so that $(h')^2<(k')^2$, whenever $k\not=0$.
Of course all derivatives of $h^2$ are equal to the corresponding derivatives of $k^2$.
We have 
\begin{align*}
\frac{\partial (h^2)}{\partial t} &\le  (h^2)''-2\langle(\nabla_SH)^\perp,(\nabla_SH)^\perp\rangle+2k^4+\widehat C(k^2+k)\\
           &\le 2(h')^2+2hh''-2\langle(\nabla_SH)^\perp,(\nabla_SH)^\perp\rangle+2k^4+\widehat C(k^2+k)\\
\end{align*}
We have 
$$(h^2)'=(k^2)'=2\langle\nabla_SH,H\rangle.$$
Since $H$ is perpendicular to $S$ we can rewrite this as
$(h^2)'=2\langle (\nabla_SH)^\perp,H\rangle$.
Hence
$$h'= \frac{\langle\nabla_SH)^\perp,H\rangle}{h}.$$
Since $|H|^2<h^2$, this implies that 
$$(h')^2\le \langle\nabla_SH)^\perp,(\nabla_SH)^\perp\rangle.$$

Plugging this in gives
$$
\frac{\partial (h^2)}{\partial t} =2h\frac{\partial h}{\partial t}\le 2hh''+2k^4+\widehat C(k^2+k).
$$
Now use the fact that $ 0\le k<h$ to rewrite this to give the modified and corrected version of Claim 19.8 on p 443:
\begin{equation}\label{epsilon19.8}
\frac{\partial h}{\partial t}  \le h''+k^3+C_1(h+1),
\end{equation} 
where $C_1=\widehat C/2$. 

We define the total length as
$$L(t)=\int |X|(x,t)dx=\int ds$$
and the total curvature 
$$\Theta(t)=\int k(x,t)|X|(x,t)dx=\int k(x,t)ds.$$

\begin{lem}{\rm (Corrected version of 19.9.)}
Let $C_2$ be an upper bound for $|Ric_g|$ over $M\times I$. Then
$$\frac{d}{dt}L\le \int(C_2-k^2)ds,$$
and
$$\frac{d}{dt}\Theta\le (C_1+C_2)\Theta+C_1L.$$
\end{lem}

\begin{proof}
The argument now follows closely in the one in the book.

The first inequality is proved on p 444 of the book. It implies that $dL/dt\le C_2L$.

For the second we consider the $\epsilon$-modified total curvature
$$\Theta_\epsilon(t)=\int h_\epsilon(x,t) ds.$$ 
We have
$$\frac{d}{dt}\Theta_\epsilon=\int\frac{\partial}{\partial t}(h_\epsilon|X|)dx =\int\left(\frac{\partial h_\epsilon}{\partial t}|X|+h_\epsilon\frac{\partial |X|}{\partial t}\right)dx.$$
Thus, 
\begin{align*}
\frac{d}{dt} \Theta_\epsilon & \le \int\bigl( h_\epsilon''+k^3+C_1(h_\epsilon+1)\bigr)ds+\int\frac{h_\epsilon}{2|X|}\bigl(-2{\rm Ric}_g(X,X)-2k^2|X|^2\bigr)dx \\
&=\int \bigl(h_\epsilon''+k^3+C_1(h_\epsilon+1)\bigr)ds-\int h_\epsilon\bigl({\rm Ric}_g(S,S)+k^2\bigr)ds\\
& \le \int\bigl(h_\epsilon''+C_1(h_\epsilon+1)-h_\epsilon({\rm Ric}_g(S,S)\bigr)ds.
\end{align*}
Since $\int h_\epsilon''ds=0$ by the fundamental theorem of calculus, we conclude that
$$\frac{d}{dt}\Theta_\epsilon(t)\le (C_1+C_2)\Theta_\epsilon(t) +C_1L(t).$$
Taking the limit as $\epsilon\mapsto 0$ gives the second statement in the lemma.
\end{proof}

It follows that  the correct version of Corollary 19.10 is $L(t)\le L(0)e^{C_2t}$, and
that we have an exponential bound on $\Theta(t)$ depending on the initial total curvature and the initial length: for example
 $$\Theta(t)\le \Theta(t)+L(t)\le \left(\Theta(0)+L(0)\right)e^{(C_1+C_2)t}.$$

\bigskip

 \noindent{\bf Effect of these corrections on the proof of Lemma 19.14 on p 447.}
 In the proof of Lemma 19.14 on p 447 there is a computation similar to the one in the proof of Claim~19.8. 
   This lemma establishes that given a Ricci flow defined on time $[t_0,t_1]$ and a ramp defined at time $t_0$, the curve shrinking flow produces a ramp defined  for all $t\in [t_0,t_1]$.  Before the start of this argument, we have already established (Corollary 19.13) that as long as the curve shrinking flow exists it produces ramps and indeed on any finite interval the minimum value of $u$ is bounded away from zero. The point is to show that if the curve shrinking flow is defined on $[t_0,t_1')$, then as we approach $t'_1$ the curvature $k$ remains bounded. We claim to do this by computing $\partial(k/u)/\partial t$
  where $u$ is the length of the projection of the unit tangent vector $S$ onto the circle direction, which of course is at most $1$. 
 This computation uses the formula in Claim 19.8 for $\partial k/\partial t$ and hence needs modification.
 
One simply replaces $k/u$ by $h/u$ and gets uses the $\epsilon$-version of Claim 19.8
(Equation~\ref{epsilon19.8}) as given above.
Following the computations on p 447 one ends up with
$$\frac{\partial}{\partial t}\left(\frac{h}{u}\right)\le \left(\frac{h}{u}\right)''+\left(\frac{2u'}{u}\right)\left(\frac{h}{u}\right)'+C'\left(\frac{h+1}{u}\right).$$

The maximum principle then gives an exponentially growing upper bound to the growth rate of the maximum of $h/u$. Since $u$ is bounded away from $0$, this gives an exponential upper bound to the growth rate of the maximum of $h$, and hence shows that $k$ stays bounded as we approach $t'_1$.

\bigskip

\noindent{\bf Effect of these corrections on Lemma 19.24.}
Lemma 19.24 is not precisely stated in that the constant $\delta >0$ and the constants $\widetilde C_i<\infty$ asserted to exist must be allowed to depend depend on the chosen family $\Gamma$, or more precisely, on  total curvature and length bounds at the initial time. With this understanding, the conclusions of this lemma hold and the proof is as given.

For any map $\Gamma\colon S^2\to \Lambda M$ as in Lemma 19.17 and any $0<\lambda<1$ the associated family $\widetilde\Gamma(c)^\lambda$, with $c\in S^2$, have uniform total curvature and length bounds that depend only on $\Gamma$. Thus, Lemma 19.24, as corrected, applies to give uniform  constants for all the curve-shrinking flows $\widetilde \Gamma(c)^\lambda$ for any $0<\lambda<1$ and any $c\in S^2$.

\bigskip

\noindent{\bf Corrections to \S 8.1}
The proof of the corrected version of  Lemma 19.24 is contained in \S 8.1 in a series of claims and then a corollary. The constants $D_0$ in Claim 19.59, $D_1$ in Claim 19.60, $D_2,D_3$ in Claim 19.61, and $D_4$ in Corollary 19.62 must be allowed to depend on the initial curvature bound, $\Theta$, and a bound, $L_0$, for the initial length, as well as on ambient Ricci flow. With this change, the conclusions
hold and, with one exception, the proofs are as given. The exception is in the inequality on the fifth line of the proof of Claim 19.61. We must add a term of the form $\int_{\gamma_t}C_1'\varphi ds$ to the right-hand side. Since the length of $\gamma_t$ is at most $re^{C_2(t_2-t')}$, this term is bounded by a constant
depending on the ambient Ricci flow and therefore can be absorbed in the constant $D_2$. Hence, this change doesn't affect the statement.


\bigskip

\begin{thebibliography}{1}
\bibitem{MT} J. Morgan and G. Tian, ``Ricci Flow and the Poincar\'e Conjecture," vol. 3 Clay Math Series,
Amer. Math Soc., Providence 2007.
\bibitem{P} G. Perelman, Finite extinction time for the solutions to the Ricci flow on certain three-manifolds. {\em preprint} arXiv.math/0307245
\end{thebibliography}
\end{document}